\documentclass[12pt]{amsart}
\usepackage{color}
\usepackage[all]{xy}
\usepackage{amssymb,amsmath}
\usepackage{amscd,amsmath,latexsym,amssymb}
\usepackage[mathcal]{euscript}
\usepackage{tikz-cd}
\usepackage[normalem]{ulem}

\date{May 27, 2020}

\calclayout

\usepackage{graphicx}

\usepackage{microtype}

\newcommand{\cB}{\mathcal B}

\newcommand{\cM}{\mathcal M}

\newcommand{\bZ}{\mathbb Z}
\newcommand{\bQ}{\mathbb Q}
\newcommand{\bC}{\mathbb C}

\newcommand{\vv}{\, | \,}

\newcommand{\iso}{\approx}

\DeclareMathOperator{\Fix}{Fix}
\DeclareMathOperator{\Sign}{Sign}

\DeclareMathOperator{\ASD}{ASD}

\DeclareMathOperator{\Spin}{Spin}

\newtheorem{dummy}{anything}[section]
\newtheorem{theorem}[dummy]{Theorem}
\newtheorem*{thma}{Theorem A}
\newtheorem*{thmb}{Theorem B}

\newtheorem{proposition}[dummy]{Proposition}
\newtheorem{corollary}[dummy]{Corollary}

\theoremstyle{definition}

  \newtheorem{example}[dummy]{Example}
  
  \newtheorem{remark}[dummy]{Remark}
   \newtheorem*{nonum}{Theorem}
   
    \newtheorem*{question}{Question}
    
  \newtheorem*{acknowledgement}{Acknowledgement}

\numberwithin{equation}{section}

\newcommand{\la}{\langle}
\newcommand{\ra}{\rangle}
\newcommand{\bd}{\partial}
\newcommand{\xa}{\hphantom{-}1}
\newcommand{\xb}{\hphantom{-}0}

\newcommand{\cy}[1]{\bZ/{#1}}

\begin{document}

\title[Cyclic Branched Coverings of Brieskorn Spheres]{Cyclic Branched Coverings of Brieskorn Spheres Bounding Acyclic $4$-Manifolds}
\author{Nima Anvari}
\address{Department of Mathematics, McMaster University
L8S 4K1, Hamilton, Ontario, Canada}
\email{anvarin@math.mcmaster.ca}

\author{Ian Hambleton}
\address{Department of Mathematics, McMaster University
L8S 4K1, Hamilton, Ontario, Canada}
\email{hambleton@mcmaster.ca}
\thanks{This research was partially supported by NSERC Discovery Grant A4000. The second author wishes to thank the Max Planck Institut f\"ur Mathematik for its hospitality and support in November 2019.}

\begin{abstract}
We show that standard cyclic actions on Brieskorn  homology $3$-spheres with non-empty fixed set do not extend smoothly to any contractible smooth $4$-manifold it may bound.  The quotient of any such extension would be an acyclic $4$-manifold with boundary a related Brieskorn homology sphere. We briefly discuss well known invariants  of homology spheres that obstruct acyclic bounding $4$-manifolds, and then use a method based on equivariant Yang-Mills moduli spaces to rule out extensions of the actions.
\end{abstract}

\maketitle

 
\section{Introduction}
In  previous work \cite{Anvari:2016} we proved that \emph{free} periodic actions on Brieskorn spheres do not extend smoothly to any bounding contractible $4$-manifold. In this paper we extend these results to the case when the periodic actions have non-empty fixed point set. This completes the solution of a problem posed by Allan Edmonds at Oberwolfach in 1988.

\smallskip
 Let $\Sigma(a,b,c)$ denote a Brieskorn homology $3$-sphere which can be realized as the boundary of a smooth contractible $4$-manifold. There are infinitely many known examples
(see Casson and Harer  \cite{CH81}, Stern \cite{S78}, Fintushel and Stern \cite{Fintushel:1981}, and  Fickle \cite{Fickle:1984}).  

As a Seifert fibred space, $\Sigma(a,b,c)$ admits a fixed-point free circle action having three orbits of finite isotropy with relatively prime orders $a, b, c$. If we restrict to a cyclic group $\cy p \subset S^1$  the action is called \textit{standard}. A standard action is \emph{free} if and only if the integer $p$ is relatively prime to $abc$. 

 In addition to the circle action there is also a natural complex conjugation action when $\Sigma(a,b,c)$ is considered as a link of a complex surface singularity. Combined with the circle action we obtain a canonical $O(2)$-action. It is known that any free action of a finite group on $\Sigma(a,b,c)$ is necessarily cyclic and standard \cite{LS92}.  
More generally, as a consequence of geometrization for $3$-manifolds:  

\begin{nonum}[{Meeks-Scott \cite{Meeks:1986}, Perelman, Boileau-Leeb-Porti \cite{Boileau:2005}, Dinkelbach-Leeb \cite{Dinkelbach:2009}}]
Any finite group action on a Brieskorn homology 
 $3$-sphere is conjugate to a subgroup of the canonical $O(2)$-action.
 \end{nonum}

In this paper we extend the results of \cite{Anvari:2016}  to include the non-free standard actions of cyclic groups, by applying our techniques to branched cyclic coverings. The main result is the following. 

\begin{thma}
The standard finite cyclic  group actions on $\Sigma=\Sigma(a,b,c)$ do not extend smoothly to any contractible $4$-manifold it may bound.
\end{thma}

\begin{remark} 
For free standard actions it is known  in many cases that they  extend locally linearly with one fixed point but
not smoothly (see  \cite[Section 4]{KL93}, \cite[Theorem B]{Anvari:2016}).   In  \cite[Theorem A]{Anvari:2016} we showed that such locally linear extensions are  never smoothable, by  using an argument based on the orientation of equivariant Yang-Mills Moduli spaces.  We expect that similar results hold for families of non-free standard actions, and that locally linear extensions exist  for families of $O(2)$-actions. 
\end{remark}

We discuss the interesting relationship between complex conjugation and Montesinos knots in Section \ref{sec;complexcong},  and show that there exist infinitely many examples for which the complex conjugation involution on $\Sigma(a,b,c)$  does extend smoothly to a smooth co-bounding contractible $4$-manifold. 

\begin{thmb}
The complex conjugation involution on the Casson-Harer families of Brieskorn homology
spheres extends smoothly to the associated contractible Mazur 4-manifolds that they bound.
\end{thmb}

 For any Brieskorn homology $3$-sphere $\Sigma(a,b,c)$, the fixed point set of complex conjugation $\tau$ is a knot which  projects to a Montesinos knot in $S^3 =  \Sigma(a,b,c)/\la \tau\ra$  (see \cite[\S 7.2]{Saveliev:1999}). For the Casson-Harer examples, the associated Montesinos knots are actually ribbon knots. The proof of  Theorem B is based on the explicit construction of Casson and Harer \cite{CH81}, in which the co-bounding Mazur manifold is a double branched cover of the $4$-ball over a ribbon disk with boundary the associated Montesinos knot. 

\begin{remark} 
 In general, we do not know when the double branched coverings of Montesinos knots bound smooth contractible (or even acyclic) $4$-manfiolds, or when the complex conjugation involution on $\Sigma(a, b,c)$ extends over a smooth contractible cobounding $4$-manifold, if one exists.  However, we observe that the answer to the extension question is not uniform. In particular, the families discovered by Stern  \cite{S78} have a different character (see \cite[Section III]{Fickle:1984} for a detailed description). In Example \ref{ex:fourfive}, we show that complex conjugation on $\Sigma(3,5,34)$ from Stern's list does not extend over any smooth $\bZ$-acyclic co-bounding $4$-manifold.
\end{remark}

\medskip
Here is a short outline of the rest of the paper. 
To prove Theorem A, it suffices to consider the standard $\cy p$-action on Brieskorn $3$-manifolds of the form $\Sigma(pa,b,c)$. Since $p$ is relatively prime to $b$ and $c$, this action is semi-free with fixed set a knot in the Brieskorn manifold.  If $W$ is a smooth contractible $4$-manifold with $\bd W = \Sigma(pa,b,c)$, then 
the quotient of the pair $(W, \Sigma)$ by any smooth extension of the standard $\cy p$-action is a smooth acyclic $4$-manifold whose boundary is the Brieskorn sphere $\Sigma(a,b,c)$. 

There are many well-known invariants that obstruct Brieskorn homology spheres bounding smooth acyclic $4$-manifolds, such as the Rokhlin invariant $\mu$, Neumann-Siebenmann invariant $\overline{\mu}$ (see \cite{Neumann:1980,Saveliev:2002a,Siebenmann:1980}, \cite[Theorem 1]{Ue:2005}), and the Fintushel-Stern $R$-invariant \cite[Theorem 1.2]{Fintushel:1985}, to name a few. In Section \ref{sec:two} we show how these invariants can rule out smooth extensions for an infinite family of examples.

In Section \ref{sec:three} we state some of the main results about equivariant Yang-Mills moduli spaces, and in Section 4 we prove a more general result which implies Theorem A (see  Theorem \ref{thm:threethree}).  Our approach (as in our previous paper) is to use equivariant Yang-Mills gauge theory (see \cite{HL92}, \cite{HL95}). The idea is that each $\Sigma(pa,b,c)$ bounds a smooth negative definite $4$-manifold $M(\Gamma)$ obtained by plumbing along a configuration $\Gamma$ of $2$-spheres. Since the $\cy p$-action is contained in the circle action on $\Sigma(pa,b,c)$, we can extend it to an action on $M(\Gamma)$ via equivariant plumbing. By combining this action together with the action on $W$ we obtain a smooth, closed, negative definite $4$-manifold $X:= M(\Gamma) \cup -W$ with a smooth, homologically trivial $\cy p$-action. The linear plumbing action contributes both fixed and invariant $2$-spheres which introduce constraints derived from the global orientation of the Yang-Mills moduli spaces. These constraints lead to a contradiction to the extension of the action over $W$. 

\begin{remark} 
Apart from actions on Brieskorn $3$-manifolds,  there are infinitely many examples of homology $3$-spheres $Q$ with free $\cy k$-actions which extend to smooth actions on contractible $4$-manifolds they bound with an isolated fixed point \cite{gordon1975knots}. These examples are constructed by gluing knot exteriors of $k$-fold branched covers of certain slice knots in the three sphere. The free $\cy k$-actions in these examples are not isotopic to the identity, in particular they are not contained in any circle action.   

 A related notion is that of a $p$-\emph{periodic} $3$-manifold. A $3$-manifold $M$ is $p$-periodic if it admits a smooth,  orientation preserving, semi-free $\cy p$-action with a circle as fixed point set. Even before geometrization it was known that if  a Brieskorn homology sphere $\Sigma(2, n, 2n-1)$, $n >2$,  is $p$-periodic for $p$ an odd prime, then  $p | n(2n-1)$ (see \cite{gilmer2002}).
\end{remark}

\begin{acknowledgement} We would like to thank Ken Baker, Brendan Owens and Nikolai Saveliev for very  helpful conversations and correspondence. We also thank the referee for constructive comments and suggestions. \end{acknowledgement}

\section{Background}\label{sec:two}
Consider the Brieskorn sphere $\Sigma(a,b,c)$ as the link of a complex surface singularity with its canonical orientation:
\[
\Sigma(a,b,c) = \{(z_1,z_2,z_3) \in \bC^3 | z_1^a + z_2^b + z_3^c = 0\}\cap S^5.
\]
The fixed-point free circle action giving $\Sigma(a,b,c)$ the structure of a Seifert fibre space is given by 
\[
t\cdot(z_1,z_2,z_3) = (t^{bc}z_1,t^{ac}z_2,t^{ab}z_3)
\]
with the \emph{standard} action being the restriction of this action to $\cy p \subset S^1.$ 
 \begin{enumerate}
\item The standard action is free if and only if $p$ is relatively prime to $abc$, and every free action is conjugate to the standard action \cite{LS92}. 
\item The quotient by the semi-free standard $\cy p$-action on $\Sigma(pa,b,c)$ 
is $\Sigma(a,b,c)$ \cite[p.~143]{Orlik:1972}. In the case $a=1$ we obtain $S^3$ as the quotient  and the branch set is the right-handed $(b,c)$ torus knot. In particular, the standard $\bZ/2$-action on $\Sigma(2,b,c)$ is given by $\tau(z_1,z_2,z_3) = (-z_1, z_2, z_3)$  making it the double branched cover of the $(b,c)$ torus knot. 
 \item In general there are topological obstructions to non-free standard actions extending locally-linearly to a smooth contractible $4$-manifold $W$ that it may bound. For example, if the standard $\bZ/2$-action on $\Sigma(2,b,c)$ extends to $W$ the fixed set is a $2$-disk that projects to a locally-flat $2$-disk in $W/\tau$ cobounding the $(b,c)$-torus knot. However, the non-trivial $(b,c)$ torus knots  are never topologically (locally-flat) slice since they have non-trivial signatures (see Rolfsen \cite[p.~218, Theorem 10]{Rolfsen:1976} and Litherland  \cite[Theorem 1]{Litherland:1979}).  
  \end{enumerate}

In the remainder of the paper we focus on smooth extensions. We conclude this section with some observations communicated to us by Nikolai Saveliev. 

\begin{example} The Brieskorn sphere $\Sigma=\Sigma(3,4,5)$ bounds a smooth contractible $4$-manifold $W$. The standard involution $\bZ/2$-action, acting via $\tau\colon \Sigma \to \Sigma$, fixes a circle corresponding to the singular fibre of order $4$, and the quotient $\Sigma/\tau=\Sigma(3,2,5)$ is the Poincar\'{e} homology sphere with branching locus the singular fibre of order $2$. We claim that the involution $\tau$ acting on  $\Sigma$ cannot extend to a smooth action on $W$. 

For if the action extends, the orbit map $W \rightarrow W/\tau $ would induce an isomorphism $H_{\ast}(W;\bQ)^{\la \tau\ra} \iso H_{\ast}(W/\tau;\bQ) $,  and the quotient $W/\tau$ would be a rational homology ball whose boundary is  the Poincar\'{e} homology sphere. By results of Smith theory, it can be shown that $W/\tau$ is in fact acyclic over the integers \cite[Theorem 5.4, p.131]{bredon1}. As is well-known, this provides a contradiction. For example, the Rokhlin invariant $\mu(\Sigma(2,3,5)) \equiv 1 \pmod 2$ and obstructs the existence of such an acyclic $4$-manifold. 
\end{example}

The above observation can be extended to an infinite family without much difficulty. Consider the family of Brieskorn homology $3$-spheres  \[ \Sigma_m = \Sigma(2m-1, 2m, 2m+1) \, \,  \text{for} \,\, m \geq 2.\]
They are part of the Casson-Harer family $\Sigma(r,rs-1,rs+1)$ with $r=2m$ and $s=1$ which bound smooth contractible $4$-manifolds $W_m$  (see Casson and Harer \cite{CH81}). The standard involution $\tau$ fixes the singular fibre of order $2m$. If the involution extends to a smooth involution on $W_m$, the quotient of $(W_m, \Sigma_m)$  by $\tau$ is another pair $(W^{\prime}_m, \Sigma^{\prime}_m)$ with $W^{\prime}_m$ acyclic and  $\Sigma_m^{\prime}=\Sigma(2m-1, m, 2m+1).$ The Seifert invariants of $\Sigma_m^{\prime}$ are given by
\[    
\Sigma^{\prime}_m(b=0; (2m-1,1), (m, -1), (2m+1, 1))
\]
which corresponds to the plumbing diagram:
\begin{figure}[ht]
\begin{center}
\[
P(\Gamma'_m) = \xymatrix{
*{\bullet}\ar@{-}[r]^<{2m-1} & *{\bullet}\ar@{-}[r]^<<{0}\ar@{-}[d] & *{\bullet} \ar@{}[]^>>>{2m+1}\\
& *{\bullet}\ar@{}[]^<{-m}
} 
\]
\end{center}
\end{figure}
\newline
The intersection form of the plumbed $4$-manifold $P(\Gamma'_m)$ 
\[
Q(\Gamma'_m)= \begin{pmatrix}
0 & 1 & 1 & 1 \\
1 & 2m-1 & 0 & 0 \\
1 & 0 & -m & 0 \\
1 & 0 & 0 & 2m+1
\end{pmatrix}
\]
is indefinite and has vanishing signature. When $m$ is even, the Wu class is given by $\omega = F_1 + F_2$ where $F_1^2=2m-1$ and $F_2^2=2m+1$. The Neumann-Siebenmann $\overline{\mu}$-invariant provides an obstruction for Brieskorn homology spheres bounding smooth acyclic $4$-manifolds and is computed as 
\[
\overline{\mu}(\Sigma^{\prime}_m) = \dfrac{\sigma(M(\Gamma))-\omega^2}{8}=\dfrac{-m}{2}.
\]
It follows that in this case, $\Sigma^{\prime}_m$ cannot have finite order in the cobordism group. In particular, it cannot bound a smooth acyclic $4$-manifold. 

In the case when $m$ is odd, the Wu class is  given by the central node in the plumbing graph $\Gamma'_m$ with $\omega^2=0$. It follows that both $\overline{\mu}$ and the Rokhlin invariant vanish and provide no information. On the other hand,  the negative-definite plumbing diagram for $\Sigma_m^{\prime}$ is obtained from $\Gamma'_m$  by  a sequence of blow-down operations on the branches with positive weights. This produces  the weight $\delta = -2$ on the central node. However,  this implies that the Fintushel-Stern invariant $R(\Sigma_m^{\prime})= -2\delta -3 = 1$ by the the Neumann-Zagier formula \cite[p.~242]{Neumann:1985}, and again it follows that $\Sigma^{\prime}_m$ cannot bound a smooth acyclic $4$-manifold. We conclude that the standard involution $\tau$ on $\Sigma(2m-1, 2m,2m+1)$ cannot extend to a smooth action on the contractible $4$-manifold $W_m$ for any $m \geq 2$. In the next section we provide a general argument using equivariant gauge theory based on our previous result \cite{Anvari:2016}.

\section{Equivariant Yang-Mills Moduli Spaces}\label{sec:three}

In this section we discuss the  theory of equivariant Yang-Mills moduli spaces. 
We first briefly summarize the results we need and refer the reader to references, see \cite{HL92}, \cite{HL95} and \cite{Anvari:2016} for further details. 

Let $P \rightarrow X$ denote a principal $SU(2)$-bundle over a closed, negative definite, smooth and simply connected  $4$-manifold $X$ with $c_2(P) =1$.  
Suppose that  $\pi=\cy p$ acts smoothly on $X$ inducing the identity on homology.  With respect to a fixed  $\pi$-invariant Riemannian metric  on $X$,  the Yang-Mills moduli space  consists of connections modulo gauge with anti-self dual ($\ASD$) curvature:
\[
\cM(X)=\{ [A] \in \cB(P) \vv F_A^+=0  \}
\]
and inherits a natural $\pi$-action. Our equivariant Yang-Mills moduli space $(\cM(X), \pi)$ is then obtained by performing an \emph{equivariant general position} perturbation of the $\ASD$ equations (see \cite[\S 2]{HL92}).

The main properties we need are summarized as follows:
\begin{itemize}
\item The equivariant moduli space $(\cM(X), \pi)$ is a Whitney stratified space with open manifold strata parametrized by isotropy subgroups.  The strata have topologically locally trivial equivariant cone bundle neighborhoods and the free stratum is a non-empty, smooth noncompact $5$-dimensional manifold consisting of irreducible ASD connections.
\item There is a $\pi$-equivariant Uhlenbeck-Taubes compactification 
\[
\overline{\cM(X)}= \cM(X) \cup X
\]
consisting of highly-concentrated $\ASD$ connections parametrized by a copy of $(X,\pi)$ in the ideal boundary with a $\pi$-equivariant collar neighbourhood diffeomorphic to $X \times (0,\lambda)$ for small $\lambda$ and trivial action on $(0,\lambda)$.
\item  There is a one-to-one correspondence between reducible ASD connections and cohomology classes $\pm u \in H^2(X;\bZ) $ with $u^2=-1$. Moreover, each reducible connection has a $\pi$-invariant neighbourhood which is equivariantly diffeomorphic to a cone over some linear action on complex projective space $\bC P^2$. 
\end{itemize}
The moduli space $(\cM(X),\pi)$ admits a canonical orientation which induces the given orientation on $X$ as a negative definite $4$-manifold in the ideal boundary and agrees with the orientation of $\bC P^2$ near the link of each reducible. We define an action $(X,\pi)$ with $\pi=\cy p$ to be oriented by fixing the negative definite orientation on $X$, and a $\pi$-equivariant $\Spin^c$-structure on $X$ for $p=2$. Additionally, 
\begin{itemize}
\item the fixed set $\Fix(\overline{\cM(X)},\pi)$ is path connected and inherits a preferred orientation from the $\pi$-action on the moduli space.
\item There exists an equivariant connected sum of linear actions on $\overline{\bC P}^2$ with the same fixed point data and tangential isotropy representations \cite{HL95}.
\item  There is a preferred choice of basis $\{e_i \}$  for $H_2(X; \bZ)$ such that $e_i \cdot e_j =\delta_{ij}$, which we call the \textit{standard} diagonal basis \cite[Definition 3.5]{Anvari:2016}.
\end{itemize}
There are restrictions on the representations of $\pi$-fixed and $\pi$-invariant spheres in terms of the standard diagonal basis. The \textit{standard orientation} on a $\pi$-invariant surface containing fixed points is the orientation induced by the action if $p$ is odd, or the orientation induced by a $\Spin^c$-structure if $p=2$.

Our assumptions imply that $X$ is homotopy equivalent to a connected sum of complex projective planes (each with the negative definite orientation). We write $X \simeq \#_1^n \overline{\bC P}^2$.
\begin{theorem}\label{thm:threeone}
Let $\pi=\cy p$, for $p$ a prime, and $(X,\pi)$ be an oriented smooth, homologically trivial action on a smooth $4$-manifold $X \simeq \#_1^n \overline{\bC P}^2$. Then
\begin{itemize}
\item[1)]    each $\pi$-fixed $2$-sphere $F \subset X$ with standard orientation has an integral homology class representation given by the expression
\[
[ F] = \sum_i a_i e_i 
\]
where each $a_i \in \{ 0,1\}$.
\item[2)] each smoothly embedded $\pi$-invariant $2$-sphere with standard orientation has integral homology class representation given by the expression
\[
[F] = \sum_i a_i e_i
\]
where $a_i \geq 0$ and $\{ e_i\}$ is the standard diagonal basis.
\end{itemize}
\end{theorem}
\begin{proof} The first part is proved in \cite[Theorem B]{HL95}, and the second  in \cite[Theorem 3.9]{Anvari:2016}.
\end{proof}

The following application of Theorem \ref{thm:threeone} is the key step in proving Theorem A. We will show that the stated assumptions on $\Sigma(pa, b,c)$ would allow the construction of $(X, \pi)$ containing a certain configuration of $\pi$-invariant $2$-spheres,  whose standard orientation is inconsistent with the preferred  orientation on $\Fix(\cM(X),\pi)$, and hence a contradiction.

\begin{corollary}\label{cor:threetwo}
Let $\pi=\cy p$, for $p$ a prime, and $(X,\pi)$ be an oriented smooth, orientation preserving, homologically trivial action on a smooth $4$-manifold $X \simeq \#_1^n \overline{\bC P}^2$.  Then there does not exist a configuration of smoothly embedded homologically non-trivial $2$-spheres $\Gamma = \{ F_1, F_2, F_3 \}$ in $(X,\pi)$, such that $F_1$ is $\pi$-fixed and $F_2, F_3$ are $\pi$-invariant, satisfying the following conditions:
\begin{itemize}
\item[1] $[F_1]^2=-1$,
\item[2] $[F_2] \cdot [F_3] = 0$ and 
\item[3] $[F_1]\cdot [F_i] =1$ for $i=2,3$.
\end{itemize}
\end{corollary}
\begin{proof}
The argument is essentially contained in \cite[Theorem 4.4]{Anvari:2016}.
\end{proof}

\section{The Proof of Theorem A} 
We now prepare for the proof of Theorem A. If 
$p$ is a prime not dividing $abc$,  then the standard $\cy p$-action on
 $\Sigma(pa,b,c)$  is semi-free with fixed set a knot $K$ in the Brieskorn manifold. 
 We suppose
 that this action extend to
 some smooth contractible $4$-manifold $W$ with $\bd W = \Sigma(pa,b,c)$. By P.~A.~Smith theory, the fixed set $\Fix(W, \cy p)$ is a $2$-disk $D$ smoothly embedded in $W$ with $\bd D = D \cap \Sigma(pa,b,c) =  K$.
 
 Recall that  $\Sigma(pa,b,c)$ bounds a smooth negative definite $4$-manifold $M(\Gamma)$ obtained by plumbing  according to a weighted graph $\Gamma$. The nodes produce a configuration of $2$-spheres, with intersections given by the edges of $\Gamma$, and normal bundles defined by the weights (see \cite[\S2]{Orlik:1972}). 
 Since the $\cy p$-action is contained in the circle action on $\Sigma(pa,b,c)$, we can extend it to an action on $M(\Gamma)$ via equivariant plumbing. By combining this action together with the action on $W$ we obtain a smooth, closed, negative definite $4$-manifold $X:= M(\Gamma) \cup -W$ with a smooth, homologically trivial $\cy p$-action.
 
 \begin{remark} 
We recall that the weight at the central node will always be $\delta = -1$ 
whenever $\Sigma(pa,b,c)$ 
 bounds an acyclic $4$-manifold (see Neumann and Zagier \cite{Neumann:1985}, Issa and McCoy \cite[Theorem 8]{Issa:2018a}). 
 \end{remark}

The linear plumbing action on $M(\Gamma)$ contributes both fixed and invariant $2$-spheres which introduce constraints derived from the global orientation of the Yang-Mills moduli spaces. In particular, the central node is always a fixed $2$-sphere. These constraints and Corollary \ref{cor:threetwo} lead to a contradiction to the extension of the action over $W$.

We need some information about the fixed-point set of the equivariant plumbing action of $\bZ/p \subset S^1$ on $M(\Gamma)$. The following proposition highlights the key feature of the non-free case, and provides additional details which have some independent interest.

\begin{proposition}\label{prop:fourtwo}
The fixed-point set of $\bZ/p \subset S^1$ on $M(\Gamma)$ contains a fixed $2$-disk,  which intersects the boundary $\Sigma(pa,b,c)$ along the singular fibre of order $pa$.
\end{proposition}
\begin{proof}
Denote the Seifert invariants of the minimal negative definite resolution by \[
\Sigma(b=-1, (pa,q_1), (b,q_2), (c,q_3))
\]
and let 
\[
\dfrac{pa}{q_1}=[w_1,...,w_n]=w_1 - \dfrac{1}{w_2 - \dfrac{1}{w_3 - \cdots}}
\]
be the continued fraction decomposition for the singular fibre of order $pa$ in $\Sigma(pa,b,c)$. Then the resolution graph has one branch where the weights are given as in Figure ~\ref{fig:resolutionDiagram}. 
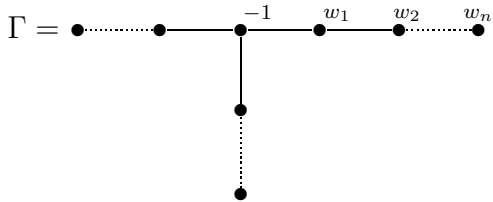
\begin{figure}[ht!]
\begin{center}
\[
\Gamma = \xymatrix{
*{\bullet}\ar@{.}[r] &  *{\bullet}\ar@{-}[r] & *{\bullet}\ar@{-}[r]^<<{-1} \ar@{-}[d] & *{\bullet} \ar@{-}[r]^<<{w_1} & *{\bullet} \ar@{.}[r]^<{w_2} & *{\bullet} \ar@{}[]^{w_n} \\
&  & *{\bullet}\ar@{.}[d] \\ & & *{\bullet}
} 
\]
\end{center}
\caption{Minimal negative definite resolution of $\Gamma$}
\label{fig:resolutionDiagram}
\end{figure}

Each node is an equivariant $D^2$-bundle over a $2$-sphere with the weights denoting the Euler number. Such a bundle is constructed by gluing two copies of $D^2 \times D^2$:
\[
\left( D^2_{+}(a_1) \times D^2(b_1) \right) \sqcup \left( D^2_{-}(a_2) \times D^2(b_2) \right),
\]
which we denote by $(a_1,b_1)$ and $(a_2,b_2)$ in base and fibre coordinates. Here $D^2(m)$ denotes a $2$-disk with action given by $\zeta^m$ with $\zeta = e^{2\pi i/p}$  the generator of the cyclic action. The rotation data are related by the linear transformation:
$$
\begin{pmatrix}
a_2 \\
b_2
\end{pmatrix} =
T(w)
\begin{pmatrix}
a_1 \\
b_1
\end{pmatrix}
, \, \text{where \,} \bigskip
T(w) = \begin{pmatrix}
-1 & 0 \\
-w & 1
\end{pmatrix}$$
with $w$ the Euler number. Equivariant plumbing begins with the fixed, central $-1$-sphere with rotation numbers $(0,1),(0,1)$ for the two copies of $D^2 \times D^2$ and extends to each of the branches so that at the $k$th node the rotation numbers are 
\[
T_k^{-1}L_k v \text{\,\, and \,\,} L_k v \text{\, \, with \,\,} v=\begin{pmatrix}
1 \\
0 
\end{pmatrix}, \,\, T_k = T(w_k)
\] 
and where $L_k = T_k A T_{k-1} \cdots T_2A T_1$ with 
\[
A=\begin{pmatrix}
0 & 1 \\
1 & 0
\end{pmatrix}
\]
which arises from interchanging fibre and base coordinates in the plumbing as we move from one node to the next. It follows that the first node at $w_1$ adjacent to the fixed central node has rotation numbers $v=(1,0)$, $T_1v=(-1,-w_1)$,  and in particular the base is always an invariant $2$-sphere. Moreover, by induction it can be shown that at the last node $L_n$ has the form:
\[
L_n = \begin{pmatrix}
r & s \\
\pm pa & \pm q_1
\end{pmatrix}
\text{\,\, for some $r,s$ }
\]
with the signs depending on whether $n$ is even or odd. In particular, $L_n v = (r, \pm pa)$ and since $\det(L_n)=\pm1$, $r \neq 0 \pmod p$ and the last node contributes the required fixed $2$-disk which  intersects  the boundary $\Sigma(pa,b,c)$ along the singular fibre of order $pa$.  Intermediate nodes on this branch may be either fixed or invariant $2$-spheres but the first and last nodes are always $\pi$-invariant (and not $\pi$-fixed).
\end{proof}

Theorem A  now follows from the following more general statement.

\begin{theorem}\label{thm:threethree}
 Let $p$ be a prime not dividing $abc$. 
Suppose the Brieskorn homology sphere $\Sigma(pa,b,c)$ bounds a smooth acyclic $4$-manifold $W$ such that $\pi_1(W)$ is the normal closure of the image of $\pi_1(\Sigma(pa,b,c))$. Then the standard $\pi=\cy p \subset S^1$-action on 
$\Sigma(pa,b,c)$ does not extend to a smooth action on $W$.
\end{theorem}
\begin{proof}
Assume the standard $\bZ/p$-action  on $\Sigma(pa,b,c)$ extends to $W$. The fixed point set is an embedded $2$-disk in $W$ with non-empty intersection with boundary $\Sigma(pa,b,c)$ along the singular fibre of order $pa$. Let $\Gamma$ denote the negative definite resolution graph for $\Sigma(pa,b,c)$. By Proposition \ref{prop:fourtwo}, the equivariant plumbing action on the plumbed $4$-manifold $M(\Gamma)$, when  restricted to $\bZ/p \subset S^1$ has a fixed $2$-disk which intersects the boundary $\Sigma(pa,b,c)$  along the singular fibre of order $pa$,  and the remaining fixed-point set in $M(\Gamma)$ consists of either isolated fixed points or fixed $2$-spheres. 

We form the simply connected smooth $4$-manifold $X=M(\Gamma) \cup_{\Sigma(pa,b,c)} (-W)$ which by construction admits a smooth, homologically trivial $\bZ/p$-action. To prove the theorem it is enough to show that $X$ always admits a configuration of spheres as in the statement of the above lemma. In fact, as in the proof of \cite[Theorem A]{Anvari:2016} for the case of free actions, the central node $F_1$ is $\pi$-fixed and its two adjacent nodes in the plumbing diagram for $M(\Gamma)$ provide the $\pi$-invariant $2$-spheres $F_2$ and $F_3$ needed to apply Corollary \ref{cor:threetwo} to get a contradiction. 
\end{proof}

\begin{example}
We demonstrate our argument with an explicit example. Let $\Sigma = \Sigma(3,4,5)=\{ x^3+y^4+z^5 =0 \}\cap S^5$ and $W$ a smooth contractible $4$-manifold with $\partial W =\Sigma$. If $\pi=\langle t \rangle \subset S^1$ denotes the involution generated by $t$, the action is given by
\begin{align*}
t  \cdot (x,yz) &= (t^{20}x,t^{15}y,t^{12}z)\\ 
                            &= (x,-y,z)
\end{align*}
with $1$-dimensional fixed point set $\{ y = 0\}$ in $\Sigma$ corresponding to the singular fibre of order $4$. The plumbing diagram for the negative definite resolution $4$-manifold $M(\Gamma)$ with boundary $\Sigma$ is given in Figure 2.
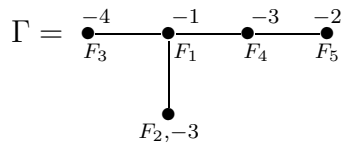
\begin{figure}[ht!]
\begin{center}
\[
\Gamma = \xymatrix{
*{\bullet}\ar@{-}[r]^<{-4}_<{F_3} & *{\bullet}\ar@{-}[r]^<<{-1}_<<{F_1} \ar@{-}[d] & *{\bullet} \ar@{-}[r]^<<{-3}_<{F_4} & *{\bullet} \ar@{}[]^<{-2}_{F_5} \\
& *{\bullet}\ar@{}[]_{F_2,  -3}
} 
\]
\end{center}
\caption{The plumbing diagram for $\Sigma(3,4,5)$.}
\end{figure}
Equivariant plumbing along $\Gamma$ produces a fixed $2$-sphere $F_1$ corresponding to the central node with square $-1$ and $3$ isolated fixed points with rotation numbers $(1,-1)$. All the other spheres in the graph $\Gamma$ are $\pi$-invariant $2$-spheres. In particular, the $\pi$-equivariant normal $D^2$-bundle over $F_3=D^2_{+} \cup D^2_{-}$ is obtained by gluing two copies of $D^2_{\pm}\times D^2$, one with rotation data $(1,0)$ and the other with $(-1,4)$. 

The former has a fixed $2$-disk which is one hemisphere of the fixed central node $2$-sphere. The latter contributes a fixed $2$-disk with non-empty intersection with the boundary $\Sigma$ corresponding to the singular fibre of order $4$. So the fixed point set in $M(\Gamma)$ has $3$ isolated fixed points, one fixed $2$-sphere with square $-1$ and a fixed $2$-disk with non-empty intersection in $\Sigma$. 

Suppose the involution on $\Sigma$ extends to a smooth involution on $W$. We combine this action with the action on $M(\Gamma)$ to obtain a smooth involution $\pi$ on $X = -W \cup M(\Gamma)$. Since this action is homologically trivial, all $2$-dimensional fixed sets must be $2$-spheres. In particular, the fixed set in $W$ must be another $2$-disk which combines with the fixed $2$-disk in $M(\Gamma)$ to contribute an additional fixed $2$-sphere $F$ in $X$. To summarize, we have
\[
\Fix(X,\pi) = \{ F_1, F, \text{and 3 isolated fixed points with rotation numbers $(1,-1)$} \}.
\]
Alternatively, the Lefschetz fixed-point formula $\chi(\Fix(X,\pi)) = \chi(X) =b_2(X)+2=7$ can be used to show $F$ must be a $2$-sphere. By the $G$-signature theorem: 
\[
\Sign(X) = -3\cot\left(\dfrac{\pi}{2}\right)\cot\left(\dfrac{-\pi}{2}\right)-\csc^2\left(\dfrac{\pi}{2}\right) + [F]^2\csc^2\left(\dfrac{\pi}{2}\right)
\]
we obtain the self intersection of the fixed $2$-sphere $[F]^2=-4$.  

 Finding the representation of $[F]$ in terms of the canonical plumbing basis leads to a system of linear equations over integers $Qx=b$ with $b=(0,0,\pm1,0,0)$ using the intersection number $[F]\cdot[F_i]=\pm1$ for $i=3$ and zero otherwise, where $Q$ is the intersection matrix  of $M(\Gamma)$ in the plumbing basis:
\[
Q (\Gamma)= \begin{pmatrix}
-1 & \xa& \xa  & \xa  & \xb \\
\xa & -3 & \xb & \xb &  \xb \\
\xa  & \xb & -4 &  \xb &  \xb \\
\xa  &  \xb & \xb & -3 & \xa \\
\xb &  \xb &  \xb & \xa  & -2
\end{pmatrix}
\]
Depending on the orientation of $F$ we obtain the representation 
$$[F] = \pm(15F_1 + 5F_2 + 4F_3+6F_4+3F_5)\in H_2(X;\bZ).$$
By Donaldson's Theorem $Q(\Gamma)$ must be diagonalizable to the identity matrix over $\bZ$:
\[
-I = C^t Q C
\]
with $C$ the change of basis matrix. The inverse gives the representation of the plumbing basis in terms of the diagonal basis.
\[
C^{-1}= \begin{pmatrix}
1& -1 & -1 & -1 &  \xb \\
0 & \xb &  \xa &   \xb & -1 \\
0 & \xa & -1 &   \xb &  \xb \\
0 &\xa &  \xb & -1 &  \xb\\
0 &  \xb & \xa & -1 & \xa
\end{pmatrix}
\]
This in turn gives the representation $[F]=\pm(e_2+e_3-e_4+e_5)$ in terms of the diagonal basis. The orientation argument implies that after applying an automorphism of the standard diagonal form if necessary (to obtain the \textit{standard } diagonal basis), the columns of $C^{-1}$ should have a consistent sign. However, the automorphisms of the diagonal form are given by permutations (row interchanges) or sign interchanges (multiplying a row by $-1$).  It is enough to observe that such a consistency cannot be obtained just examining the first three columns. 

Alternatively, the first three columns are representations of the surfaces $F_1, F_2$, and $F_3$ satisfying the conditions of the above lemma, again giving a contradiction.
\end{example}
It is not easy to decide whether a given Seifert fibred homology $3$-sphere $\Sigma$  bounds a smooth or contractible (or even acyclic) $4$-manifold. One can compute the unimodular intersection form $Q (\Gamma)$ of the minimal  negative definite resolution via the plumbing diagram. The first necessary  condition for $\Sigma$ to bound a smooth  acyclic $4$-manifold is that $Q (\Gamma)$ must be equivalent to the standard diagonal form. 

\begin{question} If $\Sigma = \bd M(\Gamma)$ admits a smooth acyclic cobounding $4$-manifold, must the diagonalizing matrix $C^{-1}$ for $Q (\Gamma)$  contain only entries $0$, and $\pm 1$~?
\end{question}
This  holds in the case of fixed 2-spheres in the equivariant setting  of Theorem \ref{thm:threeone}.
At present, there are no known examples cobounding smooth acyclic $4$-manifolds among the Seifert fibred homology $3$-spheres with more than three exceptional fibres. Indeed, any example of this type would provide a negative answer to the Montgomery-Yang problem (see \cite{Fintushel:1987} for the statement).  

\begin{example}
    We do not know whether $\Sigma(3,5,8)$ bounds a smooth contractible $4$-manifold,
     or even an acyclic $4$-manifold. However, the intersection form of the associated $M(\Gamma)$ has a new feature. 
     The plumbing graph is given as follows:
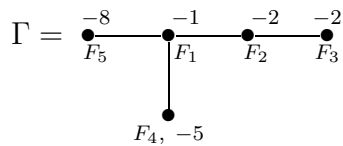
\[  \Gamma = \xymatrix{
*{\bullet}\ar@{-}[r]^<{-8}_<{F_5} & *{\bullet}\ar@{-}[r]^<<{-1}_<<{F_1} \ar@{-}[d] & *{\bullet} \ar@{-}[r]^<<{-2}_<{F_2} & *{\bullet} \ar@{}[]^<{-2}_{F_3} \\
& *{\bullet}\ar@{}[]_{F_4,\  -5}}
\] 
and the intersection matrix is: 
\[
Q (\Gamma)= \begin{pmatrix}
-1 & \xa & \xb & \xa & \xa \\
\xa & -2 & \xa & \xb & \xb \\
\xb & \xa & -2 & \xb & \xb \\
\xa & \xb & \xb & -5 & \xb \\
\xa & \xb & \xb & \xb & -8
\end{pmatrix}
\]
After diagonalizing, so that
$C^t Q(\Gamma) C= -I$, we have:
\[
C^{-1}= \begin{pmatrix}
\xa & -1 & \xb & -1 &  -1 \\
\xb &  \xa & -1 &  -1 & -1 \\
\xb &  \xb & \xa &  -1 & -1 \\
\xb &  \xb &  \xb & -1 & \hphantom{-}2 \\
\xb & \xb & \xb & -1 & \xa
\end{pmatrix}
\]
In previously computed examples with three exceptional fibres, the diagonalizing matrix $C^{-1}$ expressing the plumbing basis in terms of a standard basis only had entries $0$, and $\pm 1$.  
The new feature in this case is the presence of the off-diagonal entry 2.  The analogous calculation for lens spaces bounding rational homology balls was studied by  Lisca (see \cite[Proposition 3.3, Proposition  5.2]{Lisca:2007}). 
 \end{example}

\section{Complex Conjugation and Montesinos Knots}\label{sec;complexcong}
  Apart from the standard finite cyclic group actions on $\Sigma(a,b,c)$, we also have the complex conjugation involution, given by the formula
  $\tau(z_1, z_2, z_3) = (\bar z_1, \bar z_2, \bar z_3)$. Then $K:= \Fix (\tau)$ is an embedded circle, and $\Sigma/\la \tau\ra = S^3$. The image of $K \subset S^3$ under the quotient map is a \emph{Montesinos knot} (see Saveliev \cite[Section 7.2. p.~79]{Saveliev:1999}, Lecuona \cite[Section 2]{Lecuona:2012}). 
By combining the standard circle action with $\la \tau\ra$, we obtain a canonical $O(2)$-action.
  One consequence of geometrization:
  
 \begin{nonum}[{Meeks-Scott \cite{Meeks:1986}, Perelman, Boileau-Leeb-Porti \cite{Boileau:2005}, Dinkelbach-Leeb \cite{Dinkelbach:2009}}] Any finite group action on a Brieskorn homology 
 $3$-sphere is conjugate to a subgroup of the canonical $O(2)$-action.
 \end{nonum}
   
  Since the actions contained in the circle subgroup are handled by Theorem A, it remains to discuss the complex conjugation involution.  
\begin{example}
   Akbulut and Kirby \cite{Akbulut:1979,Akbulut:1979a}  showed that $\Sigma(2,5,7) = \bd W$, where $W$ is a certain contractible, smooth $4$-manifold (a Mazur manifold). In fact, the Mazur manifold could be defined as the double branched cover over a ribbon disk. 
  \end{example}

   \begin{example}[Saveliev] The complex conjugation involution on $\Sigma(2,5,7)$ extends 
   smoothly to the Mazur manifold $W$ that it bounds.
   \end{example}

   \begin{proof}(Sketch, based on correspondence from Nikolai Saveliev)
   The Mazur manifold bounding $\Sigma(2,5,7)$ can be obtained by surgery on a strongly invertible two-component link (see \cite{Akbulut:1979}).  One can check that the corresponding involution makes $\Sigma(2,5,7)$ into a double branched cover of $S^3$ with branch set a ribbon knot, and the Mazur manifold into a double branched cover of $D^4$ with branch set the ribbon disk. That this involution is in fact the complex conjugation involution on $\Sigma(2,5,7)$ follows by matching the two branch sets, the ribbon knot and the Montesinos knot. The latter can be done using SnapPea. 
   \end{proof}
   
   \begin{remark} The identification of the involution in this argument as complex conjugation is also a consequence of geometrization, since the fixed set of the standard involution on $\Sigma(2,5,7)$ is the torus knot $(5,7)$. 
   \end{remark}
   
   \begin{example}
   The example of $\Sigma(2,5,7)$ is obtained by using a cancelling pair of 1 and 2 handles with the 2-handle attached along a strongly invertible knot in $S^1 \times S^2$.  Additional examples can be constructed as follows. Take the same handle diagram for the Mazur manifold that gives $\Sigma(2,5,7)$ as the boundary but change the framing of the 2-handle from $p=3$ to $ p=2, ~4$ to get Mazur manifolds with boundary $\Sigma(2,3,13)$ and $\Sigma(3,4,5)$ respectively (see the comment by Saveliev  \cite[p.~190,  Example 7.11]{Saveliev:2002}). The link is still strongly invertible and the proof above for extending the complex conjugation should go through as before. To determine the framings we can use the diffeomorphism formulas in Akbulut and Kirby \cite[Proposition 1, p.~261]{Akbulut:1979}. For example,  in the notation of \cite{Akbulut:1979} $$\partial W^{-}(0,2) = \partial W^{+}(2,-1) = \partial W^{+}(1,0) = \Sigma(2,3,13).$$ The Kirby diagram for $W^{-}(0,2)$ is the same for the Mazur manifold with boundary $\Sigma(2,5,7)$ except the framing is $p=2$ instead of 3.
\end{example}

  Analysing the construction of Casson-Harer \cite{CH81} \cite{Fickle:1984}, combined with geometrization, shows that the complex conjugation action on the infinite families of  Brieskorn spheres
  \begin{enumerate}
  \item $\Sigma(p, ps-1, ps+1)$, with $p$ even, $s$ odd, and 
  \item $\Sigma(p, ps\pm1, ps\pm2)$ with $p$ odd, $s$ arbitrary
  \end{enumerate}
   extends to their associated co-bounding Mazur manifolds.  The example $\Sigma(2,5,7)$  above is the case $p=2$ and $s = 3$. 

\begin{thmb}
The complex conjugation involution on the Casson-Harer families (i) and (ii)  of Brieskorn homology
spheres extends smoothly to the associated contractible Mazur 4-manifolds that they bound.
\end{thmb}
\begin{proof}
Let $\mathcal{C}$ denote the class of Montesinos knots whose double branched covers are Brieskorn homology $3$-spheres $\Sigma(a,b,c)$,  and satisfy the property that a single band or ribbon move converts it to a $2$-component unlink. If $K$ is a knot in this class, then such a ribbon move gives a cobordism between $K$ and the $2$-component unlink in $S^3 \times [0,1]$. Capping-off the unlink components with $2$-disks gives an embedded ribbon disk whose double branched cover is a $4$-manifold $W$ diffeomorphic to $S^1 \times B^3$ with a $2$-handle attached (see the construction in  \cite[p.~30]{CH81}). Turning the handle decomposition of $W$ upside down gives a Mazur manifold co-bounding $\Sigma(a,b,c)$. The restriction of the involution on $W$ to the boundary $\Sigma(a,b,c)$ must be complex conjugation by geometrization, and the Brieskorn homology spheres that arise from the class $\mathcal{C}$ include   the Casson-Harer families (i) and (ii) listed above. 
\end{proof}

The following example shows that complex conjugation on $\Sigma(a,b,c)$ doesn't always extend to a co-bounding contractible $4$-manifold when one exists.

\begin{example}\label{ex:fourfive}  The Brieskorn homology $3$-sphere $\Sigma(3,5,34)$ in Stern's list 
$$\Sigma(r, rs\pm2,2r(rs\pm2)+rs\pm1)$$
 with $r=3$  and $s=1$ bounds a smooth contractible $4$-manifold $W$, but the projection in $S^3$ of  the fixed point set of the complex conjugation involution $\tau$ in this case is a Montesinos knot $K$ \emph{which is not slice} since it does not satisfy the Fox-Milnor condition on its Alexander polynomial 
   $$\Delta_{K}(t) = 2 - 5t - 2t^2 + 11t^3 - 2t^4 - 5t^5 + 2t^6$$ (see the remark \cite[p.~4]{Issa:2018}). If the involution $\tau$ were to extend over $W$, then $K \subset S^3$ would bound an embedded $2$-disk in the the acyclic smooth $4$-manifold $W/\la \tau \ra$, contradicting the Fox-Milnor theorem \cite{Fox:1966}. 
\end{example}

\providecommand{\bysame}{\leavevmode\hbox to3em{\hrulefill}\thinspace}
\providecommand{\MR}{\relax\ifhmode\unskip\space\fi MR }
\providecommand{\MRhref}[2]{%
  \href{http://www.ams.org/mathscinet-getitem?mr=#1}{#2}
}
\providecommand{\href}[2]{#2}

\end{document}